\documentclass[12pt]{article}
\usepackage{amsthm}
\usepackage{amsmath}
\usepackage{enumerate}
\usepackage{amssymb}
\usepackage{latexsym}
\usepackage{amsfonts}
\usepackage{color}
\usepackage{mathrsfs}
\usepackage{epsfig}
\newtheorem{theorem}{\bf Theorem}[section]

\newtheorem{definition}[theorem]{\bf Definition}

\newtheorem{lemma}[theorem]{\bf Lemma}

\newsavebox{\savepar}

\begin{document}
		
		\title{Elliptic Partial Differential Equation Involving a Singularity and a Radon measure}
		\author{Akasmika Panda, Sekhar Ghosh \& Debajyoti Choudhuri\footnote{Corresponding
				author: dc.iit12@gmail.com}\\
				{\small Department of Mathematics, National Institute of Technology Rourkela, India}\\ {\small Emails:  akasmika44@gmail.com, sekharghosh1234@gmail.com}}
		\date{}
		\maketitle
		
		\begin{abstract}
			\noindent The aim of this paper is to prove the existence of solution for a partial differential equation involving a singularity with a general nonnegative, Radon measure $\mu$ as its nonhomogenous term which is given as
			\begin{eqnarray}
				-\Delta u&=& f(x)h(u)+\mu~\text{in}~\Omega,\nonumber\\
				u&=&0~\text{on}~\partial\Omega,\nonumber\\
				u&>& 0~\text{on}~\Omega\nonumber,
			\end{eqnarray}
			where $\Omega$ is a bounded domain of $\mathbb{R}^N$, $f$ is a nonnegative function over $\Omega$.
			\\
			{\bf keywords}:~Elliptic PDE; Sobolev space; Schauder fixed point theorem.
		\end{abstract}
		\section{Introduction}
		Problems involving singularity have of late become a hugely popular interest of research in the Mathematical community. A good amount of research has been done to prove the existence of a solution to the problem
		\begin{eqnarray}
			-\Delta u&=& f(x)h(u)~\text{in}~\Omega,\nonumber\\
			u&=&0~\text{on}~\partial\Omega.\label{ieq0}
			%	u&>& 0~\text{on}~\Omega,
		\end{eqnarray}
		A few noteworthy results on such a problem with $\Omega\subset\mathbb{R}^N$ being a bounded domain can be found in \cite{lazer}, \cite{giach1, giach2}, \cite{bocca}, \cite{tali}, \cite{gati}, \cite{arcoya}, \cite{Oliva h} and the references therein. An existence and uniqueness result has been proved by Lazer and McKenna \cite{lazer}, pertaining to the case $h(s)=\frac{1}{s^{\gamma}}$, with $f$ being a H\"{o}lder continuous function. The authors in \cite{lazer} arrived at the unicity of solution by the application of the sub-super solution method. Furthermore, the authors in \cite{lazer} have proved that the problem possesses a solution iff $\gamma<3$. They have also shown that for $\gamma>1$, solutions to the problem with infinite energy exists.
		%A few generalizations to this result can be found in \cite{lazer}.
		A weaker condition on the function $f$ can be considered by picking $f$ from $L^p(\Omega)$, for $p \geq 1$, or from the space of Radon measures. Boccardo and Orsina in \cite{bocca} have proved the existence and uniqueness of a solution to a similar  problem as in \cite{lazer} but with lesser assumptions of regularity on $f$. They have considered $f\geq0$ in $\Omega$, $\gamma>0$. The existence result depends on the $L^p$ space from where $f$ has been chosen. The value of $\gamma$ also decides the space in which the solution belongs to, i.e. if $\gamma < 1$ then $u\in W_0^{1,1}(\Omega)$, if $\gamma=1$ then $u\in H_0^1(\Omega)$ and if $\gamma>1$ then $u\in H_{loc}^1(\Omega)$ where the zero Dirichlet boundary condition has been assumed in a weaker sense than the usual sense of trace. When $f$ is a bounded, Radon measure the problem may not possess a solution in general and in this case the question of nonexistence is of great importance as seen in \cite{bocca}. In \cite{cave} the authors have considered a nonlinear elliptic boundary value problem with a general singular lower order term. Here the authors have proved the existence of a distributional solution. A slight improvement of the result in \cite{lazer}, can be found in \cite{yijing}. In \cite{canino1}, a minimax method has been used to address the `jumping problem' for a singular semilinear elliptic equation. A symmetry of solutions have been shown in \cite{canino2}, for a class of semilinear equations with singular nonlinearities. In \cite{canino3}, the authors have considered quasilinear elliptic equations involving the $p$-Laplacian and singular nonlinearities. They have deduced a few comparison principles and have proved some uniqueness results. The reader may also refer to a series of noteworthy contributions made by Canino et al. in \cite{tali}, \cite{gati}, \cite{arcoya} to the semilinear elliptic problem with a singularity. It is worth mentioning the work due to Giachetti et al. \cite{giach1, giach2} and the references therein. In this paper we will prove the existence of weak solution to the following PDE.
		\begin{eqnarray}\label{ieq1}
			-\Delta u&=& f(x)h(u)+\mu~\text{in}~\Omega,\nonumber\\
			u&=&0~\text{on}~\partial\Omega,\nonumber\\
			u&>& 0~\text{on}~\Omega,
		\end{eqnarray}
		where $\Omega$ is a bounded domain in $\mathbb{R}^N$ for $N \geq 2$, $f $ is a nonnegative function and $\mu$ is a nonnegative, bounded Radon measure. We will further assume greater regularity on $f$ in $\eqref{ieq1}$ to guarantee the existence of a very weak solution.
		\subsection{Notations}
		%\subsection{Notation}
		This subsection is about the notations and definitions which will be used throughout this article. Henceforth, we will denote by $\Omega$ a bounded domain in $\mathbb{R}^N$. The Sobolev space denoted by $W^{k,p}(\Omega)$ $\cite{Evans}$ consists of all locally summable functions $u: \Omega\rightarrow \mathbb{R}$ such that for each multiindex $\alpha$ with $|\alpha|\leq k$, $D^{\alpha}u$ exists in the weak sense and belongs to $L^p(\Omega)$.	If $u\in W^{k,p}(\Omega)$, we define its norm as
		\begin{center}
			$\|u\|_{W^{k,p}(\Omega)}= 	{	\left\{
				\begin{array}{ll}
				\left(\sum_{|\alpha|\leq k}\int_\Omega |{D^{\alpha}u}|^p dx\right)^{\frac{1}{p}} &(1\leq p<\infty),\\
				\sum_{|\alpha|\leq k}\|D^{\alpha}u\|_{L^{\infty}(\Omega)} & (p=\infty).
				\end{array}
				\right. }$
		\end{center}
	 The closure of $C_c^{\infty}(\Omega)$ in $W^{k,p}(\Omega)$ is denoted by $W_0^{k,p}(\Omega)$. The local Sobolev space, $W_{loc}^{k,p}(\Omega)$, is defined to be the set of functions $u$ such that for every compact subset $K$ of $\Omega$ $u\in W^{k,p}(K)$. The H\"{o}lder Space $C^{k,\beta}(\bar{\Omega})$ with  $0<\beta\leq1$ $\cite{Evans}$, consists of all those functions $u\in C^k(\bar{\Omega})$ such that $\sum\limits_{|\alpha|\leq k}\sup|D^{\alpha}u|+\sup\limits_{x\neq y} \left\{\frac{|D^ku(x)-D^ku(y)|}{|x-y|^{\beta}}\right\}$ is finite. We will use the following truncation functions. For fixed $k > 0$
		$$T_k(s)=\max\{-k,\min\{k,s\}\}$$ and $$G_k(s)=(|s|-k)^+sign(s)$$ with $s\in \mathbb{R}$. It is easy to observe that $T_k(s)+G_k(s)=s$, for any $s\in \mathbb{R}$ and $k>0$.\\
		%We will signify $C$ by some constants, whose value will be independent of the indexes of the sequence and can be different for each lines.
		We will denote the space of all finite Radon measures on $\Omega$ as $\mathcal{M}(\Omega)$ endowed with the `total variation norm' which is defined as  $$\|\mu\|_{\mathcal{M}(\Omega)}=\int_{\Omega}d|\mu|.$$
	The Marcinkiewicz space  ${M}^q(\Omega)$ $\cite{Benilan}$ (or the weak $L^q(\Omega)$ space), $0 < q <\infty$, is the set of all measurable functions $f:\Omega\rightarrow \mathbb{R}$ such that the corresponding distribution satisfies an estimate of the form
		$$m(\{x\in \Omega:|f(x)|>t\})\leq \frac{C}{t^q},~t>0,~C<\infty$$, where $m$ is the Lebesgue measure. 
	We also have ${M}^q(\Omega)\subset {M}^{\bar{q}}(\Omega)$ if  $q\geq \bar{q}$, for some fixed positive $\bar{q}$. We recall here the following useful continuous embeddings
		\begin{equation}\label{mar}
			L^q(\Omega)\hookrightarrow {M}^q(\Omega)\hookrightarrow L^{q-\epsilon}(\Omega),
		\end{equation}
		for every $1<q<\infty$ and $0<\epsilon<q-1$. The article is organized as follows. In Section $2$ we state and prove the main results pertaining to the cases $0<\gamma< 1$ and $\gamma\geq 1$. In Section $3$ we will prove the existence of a solution to $\eqref{ieq1}$ in the very weak sense for $0<\gamma<1$ with two different regularity assumptions on $f$. To the end of the article, in the Appendix, we will derive a Kato type inequality corresponding to the problem defined in section 3.
		
		\section{Assumptions, Definitions and the main results}
		We consider the following boundary value problem.	
		\begin{eqnarray}\label{eq0}
			-\Delta u&=& h(u)f+\mu ~\text{in}~\Omega,\nonumber\\
			u &=& 0 ~\text{on}~ \partial\Omega,\\
			u &>& 0 ~\text{in}~ \Omega,\nonumber
		\end{eqnarray}	
		where $N>2$, $\mu$ is a nonnegative, bounded, Radon measure on $\Omega$,  $f$ is a nonnegative function in $ L^m (\Omega)$ for $m \geq 1$, which could be a measure.\\	
		The function $h : \mathbb{R}^+ \rightarrow \mathbb{R}^+$ is a nonlinear, nonincreasing, continuous function such that
		\begin{equation}\label{eqn3}
			\underset{s\rightarrow 0^+}{\lim}~ h(s)\in (0,\infty]~ \text{and}~\underset{s\rightarrow \infty}{\lim} h(s)=h(\infty)<\infty
		\end{equation}
	with the following growth condition near zero and infinity
		\begin{equation}\label{eqq1}
		\exists\, C_1, \underline{K}>0~\text{such that}~ h(s)\leq\frac{C_1}{s^\gamma}~\text{if}~s<\underline{K},	\gamma> 0,
		\end{equation}
		\begin{equation}\label{eqqq}
		\exists\, C_2, \overline{K}>0~\text{such that}~ h(s)\leq\frac{C_2}{s^\theta}~\text{if}~s>\overline{K}, \theta>0
		\end{equation}
	respectively. We will later observe that the behavior of $h$ at infinity influences the regularity of the solution $u$. Now we give two important definitions which is essential to our study of the problem in ($\ref{eq0}$).
		\begin{definition}\label{defn-1}
			Let $(\mu_n)$ be the sequence of measurable functions in $\mathcal{M}(\Omega)$. We say $(\mu_n)$ converges to $\mu\in \mathcal{M}(\Omega)$ in the sense of measure i.e. $\mu_n\rightharpoonup \mu$ in $\mathcal{M}(\Omega)$, if
			$$\int_\Omega f d\mu_n\rightarrow \int_\Omega f d\mu,~\forall f\in C_0(\Omega).$$
			%This sense of convergence in the measure space is also called as the weak sense of convergence in measure.
		\end{definition}
		\begin{definition}
			If $0<\gamma<1$, then a weak solution to the problem in ($\ref{eq0}$) is a function in $W_0^{1,1}(\Omega)$ such that
			\begin{eqnarray}\label{cond1}
				\int_{\Omega}\nabla u.\nabla\varphi&=&\int_{\Omega}fh(u)\varphi +\int_{\Omega}\varphi d\mu,~ \forall\varphi\in C_c^1(\bar\Omega)
			\end{eqnarray}
			and \begin{eqnarray}\label{cond2}\forall K\subset\subset\Omega, ~\exists~C_{K}~\text{such that}~u\geq C_{K}>0.\end{eqnarray} If $\gamma \geq 1$, then a weak solution to the problem is a function $u\in W_{loc}^{1,1}(\Omega)$ satisfying ($\ref{cond1}$) and ($\ref{cond2}$) such that $T_{k}^{\frac{\gamma+1}{2}}(u)\in W_0^{1,2}(\Omega)$ for each fixed $k>0$.
			%with $fh\in L_{loc}^1(\Omega)$.
		\end{definition}
		% 	\begin{definition}\label{defn0}
		%	If $\gamma\leq1$, then a weak solution $\cite{Oliva}$ to the problem $\eqref{eq0}$ is a function $u\in W_0^{1,1}(\Omega)$ such that
		%	\begin{equation}\label{eqn1}
		%	\forall K \subset\subset \Omega \hspace{0.2cm} \exists \hspace{0.1cm}C_K : u\geq C_K>0,
		%\end{equation}
		%and such that
		%\begin{equation}\label{eqn2}
		%\int_{\Omega}\nabla u.\nabla\varphi=\int_{\Omega}h(u)f\varphi +\int_{\Omega}\varphi d\mu, \hspace{0.4cm}\forall\varphi\in C_c^1(\Omega).
		%\end{equation}	and $h(u)f\in L_{loc}^1(\Omega).$			
		%	If $\gamma>1$, then a weak solution to the problem $\eqref{eq0}$ is a function $u\in W_{loc}^{1,1}(\Omega)$ satisfying $\eqref{eqn1}$ and $\eqref{eqn2}$, and such that $T_k^{\frac{\gamma+1}{2}}(u)\in W_0^{1,2}(\Omega)$ for every fixed $k>0$.
		%\end{definition}
		\noindent In the subsections $\ref{sub1}$ and $\ref{sub2}$, we will prove the existence of solution to the problem ($\ref{eq0}$) for both the cases, i.e. $0<\gamma<1$ and $\gamma\geq 1$. We begin with the following  sequence of problems.
		\begin{eqnarray}\label{eq1}	
			-\Delta u_n&=& h_n\left(u_n+\frac{1}{n}\right)f_n+\mu_n ~\text{in}~\Omega,\nonumber\\
			u_n &=& 0 ~\text{on}~ \partial\Omega,	
		\end{eqnarray}
		where ($\mu_n$) is a sequence of smooth nonnegative functions bounded in $L^1(\Omega)$ that converges to $\mu$ in the sense of Definition $\ref{defn-1}$. Further, $h_n=T_n(h)$ and $f_n=T_n(f)$ are the truncations at level $n$. The weak formulation of $\eqref{eq1}$ is
		\begin{equation}\label{weak}
			\int_{\Omega} \nabla u_n  \nabla\varphi = \int_{\Omega}  h_n\left(u_n+\frac{1}{n}\right)f_n\varphi + \int_\Omega \mu_n\varphi, ~\forall\varphi\in C_c^1(\bar\Omega).
		\end{equation}
		We now prove the existence of a solution to the problem ($\ref{eq1}$) in the following lemma.
		\begin{lemma}\label{13'}
			The problem $\eqref{eq1}$ admits a nonnegative weak solution $u_n\in W_0^{1,2} (\Omega)\cap L^{\infty}(\Omega)$.
			\begin{proof}
				We will apply the Schauder's fixed point argument to prove the lemma. For a fixed $n\in \mathbb{N}$ let us define a map, $$G:L^2(\Omega)\rightarrow L^2(\Omega)$$
				such that, for any $v\in L^2(\Omega)$ we get a unique weak solution $w$ to the following problem
				\begin{eqnarray}	
					-\Delta w&=& h_n\left(|v|+\frac{1}{n}\right)f_n+\mu_n ~\text{in}~\Omega,\nonumber\\
					w &=& 0 ~\text{on}~ \partial\Omega.\label{schauder} 	
				\end{eqnarray}
				The existence of a unique $w\in W_0^{1,2}(\Omega)$ corresponding to a $v\in L^2(\Omega)$ is guaranteed due to the Lax-Milgram theorem. Thus we can choose $w$ as a test function in the weak formulation of ($\ref{schauder}$) with the test function space $W_0^{1,2}(\Omega)$. Let $\lambda_1$ be the first eigenvalue of ($-\Delta$). On using the Poincar\'{e} inequality we get
				\begin{align}\label{equation}
					\lambda_1\int _\Omega |w|^2 &\leq\int _\Omega |\nabla w|^2
					% \int _\Omega \nabla w .\nabla w
					\nonumber\\& = {\int _\Omega h_n\left(|v|+\frac{1}{n}\right)f_n w} +\int _\Omega w{\mu}_n~~\text{(by the weak formulation of (\ref{schauder}))}\nonumber
					\\& \leq{ C_1 \int _{(|v|+\frac{1}{n}< \underline{K})}\frac{f_n w} {(|v|+\frac{1}{n})^\gamma}} + \max_{[\underline{K},\overline{K}]} h(s) \int_{(\underline{K}\leq (|v|+\frac{1}{n})\leq \overline{K})} f_n w \nonumber\\& ~~~~~~+ C_2 \int _{(|v|+\frac{1}{n} > \overline{K})}\frac{f_n w} {(|v|+\frac{1}{n})^\theta} + C(n) \int_\Omega|w|\nonumber\\
					& \leq { C_1 n^{1+\gamma}\int _{(|v|+\frac{1}{n}< \underline{K})}|w| + n \max_{[\underline{K},\overline{K}]} h(s) \int_{(\underline{K}\leq (|v|+\frac{1}{n})\leq \overline{K})}  |w|  + C_2n^{1+\theta} \int _{(|v|+\frac{1}{n} > \overline{K})} |w| }\nonumber
					\\& ~~~~~~+ C(n) \int_\Omega|w|\nonumber
					\\& \leq C(n,\gamma) \int _\Omega |w|\nonumber\\
					&\leq C'.C(n,\gamma)\|w\|_2~~\text{(by using the H\"{o}lder's inequality)}.
				\end{align}
				%So, by using the Poincar\'{e} inequality on the left hand side and H\"{o}lder inequality on the right side of the  inequality $\eqref{equation}$ we get, $$\int_\Omega |w|^2 \leq C'.C(n,\gamma)\bigg(\int_\Omega |w|^2\bigg)^\frac{1}{2}.$$
				This shows that
				\begin{equation}\label{eq}
					\|w\|_{L^2(\Omega)}\leq C'.C(n,\gamma),
				\end{equation}
				where $C'$ and $C(n,\gamma)$ are independent of $v$. We will prove that the map $G$ is continuous over $L^2(\Omega)$. For this let us consider a sequence ($v_k$) that converges to $v$ with respect to the $L^2$-norm. By the dominated convergence theorem we obtain
				\begin{center}
					$\|\big(h_n\left(v_k +\frac{1}{n}\right)f_n+{\mu}_n\big)-\big(h_n\left(v +\frac{1}{n}\right)f_n+{\mu}_n\big)\|_{L^2(\Omega)}\longrightarrow 0.$
				\end{center}	
	Hence, by the uniqueness of the weak solution, we can say that $w_k=G(v_k)$ converges to $w=G(v)$ in $L^2(\Omega)$. Thus $G$ is continuous over $L^2(\Omega)$.\\
{\it Claim:}  $G(L^2(\Omega))$ is relatively compact in $L^2(\Omega)$.\\
	  We have proved in $\eqref{eq}$ that
	  \begin{align}
	  \int_\Omega |\nabla w|^2&=\int _\Omega |\nabla G(v)|^2\nonumber\\
	  & \leq C'.C(n,\gamma),\nonumber
	  \end{align}
				for any $v\in L^2(\Omega)$, so that, $G(L^2(\Omega))$ is relatively compact in $L^2(\Omega)$ by the Rellich-Kondrachov theorem. This proves that $G(L^2(\Omega))$ is relatively compact in $L^2(\Omega)$. Hence the claim.\\
	 Therefore, on applying the Schauder fixed point theorem to $G$ we guarantee an existence of a fixed point $u_n\in L^2(\Omega)$ that is a weak solution to $\eqref{eq1}$ in $W_0^{1,2}(\Omega)$. \\
	 Since  $\left(h_n\left(u_n+\frac{1}{n}\right)f_n+{\mu}_n\right) \geq 0$, hence by the maximum principle $u_n\geq 0$.
				Furthermore, for a fixed $n$, we have $u_n$ belongs to $L^{\infty}(\Omega)$ (by Th\'{e}or\`{e}me 4.2, page 215 in $\cite{Stampacchia}$) because the righthand side of ($\ref{eq1}$) is in $L^{\infty}(\Omega)$ and this concludes the proof.				 
			\end{proof}
		\end{lemma}
		\noindent The next step is to prove that the sequence ($u_n$) is uniformly bounded from below on every compact subset of $\Omega$.
		\begin{lemma}\label{l4}
			For every $K\subset\subset \Omega$ there exists $C_K$ such that $u_n(x)\geq C_K >0$, a.e. in $K$, for every $n\in \mathbb{N}$.
			\begin{proof} Let us consider the sequence of problems
				\begin{eqnarray}\label{eq2}	
					-\Delta v_n&=& h_n\left(v_n+\frac{1}{n}\right)f_n ~\text{in}~\Omega,\nonumber\\
					v_n &=& 0 ~\text{on}~ \partial\Omega.	
				\end{eqnarray}
				We first show the existence of a weak solution $v_n$ to the problem in  ($\ref{eq2}$) such that for every $ K \subset \subset \Omega,~\text{there exists}~ C_K~ \text{such that}~ v_n \geq C_K >0,$
				for almost every $x$ in $K$, $C_K$ being independent of $n$. The existence of a weak solution to ($\ref{eq2}$) follows from the same proof as in Lemma $\ref{13'}$. Since $0\leq f_n\leq f_{n+1}$ and $h$ is nonincreasing we have that $h_n$ is nonincreasing. Thus
				\begin{align}\label{monotonicity1}
					-\Delta v_n&=f_n h_n\left(v_n+\frac{1}{n}\right)\nonumber\\
					&\leq  f_{n+1}h_n\left(v_n+\frac{1}{n+1}\right)~ \text{in}~\Omega,\nonumber\\	
				v_n	&= 0~ \text{on}~\partial\Omega.
				\end{align}
				We also know that $v_{n+1}$ is a weak solution to
				\begin{eqnarray}\label{monotonicity2}
					-\Delta v_{n+1}&=& f_{n+1}h_{n+1}\left(v_{n+1}+\frac{1}{n+1}\right) ~\text{in}~\Omega,\nonumber\\
					v_{n+1} &=& 0 ~\text{on}~ \partial\Omega.
				\end{eqnarray}
				The difference between the weak formulations of the problems in ($\ref{monotonicity1}$), ($\ref{monotonicity2}$) with the choice of a test function as $(v_n-v_{n+1})^{+}$ we obtain
				
				\begin{eqnarray}\label{monotonicity3}
					\int_{\Omega}\nabla(v_n-v_{n+1})\cdot\nabla(v_n-v_{n+1})^{+}&=&\int_{\Omega}|\nabla(v_n-v_{n+1})^{+}|^2\nonumber\\&\leq &\int_{\Omega}f_{n+1}\left[h_n\left(v_n+\frac{1}{n+1}\right)\right.\nonumber\\
					& &\left.-h_{n+1}\left(v_{n+1}+\frac{1}{n+1}\right)\right](v_n-v_{n+1})^{+}\nonumber\\
					&= & \int_{\Omega}f_{n+1}\left[\left\{h_n\left(v_n+\frac{1}{n+1}\right)\right.\right.\nonumber\\
					& &\left.\left.-h_{n}\left(v_{n+1}+\frac{1}{n+1}\right)\right.\right\}\chi_{[v_n\leq v_{n+1}]}(v_n-v_{n+1})^{+}\nonumber\\
					& &+\left\{h_n\left(v_n+\frac{1}{n+1}\right)\right.\nonumber\\
					& &\left.\left.\left.-h_{n}\left(v_{n+1}+\frac{1}{n+1}\right)\right.\right\}\chi_{[v_n> v_{n+1}]}(v_n-v_{n+1})^{+}\right]\nonumber\\
					& = & \int_{\Omega}f_{n+1}\left\{h_n\left(v_n+\frac{1}{n+1}\right)\right.\nonumber\\
					& &\left.\left.-h_{n}\left(v_{n+1}+\frac{1}{n+1}\right)\right.\right\}\chi_{[v_n>v_{n+1}]}(v_n-v_{n+1})^{+}\nonumber\\
					&\leq & 0.
				\end{eqnarray}
				Therefore, $(v_n-v_{n+1})^{+}=0$ almost everywhere in $\Omega$, thereby implying that $v_n \leq v_{n+1}$. We again use the Th\'{e}or\`{e}me 4.2, in page 215 \cite{Stampacchia} to obtain
				\begin{eqnarray}\label{boundedness}
					\|v_1\|_{\infty} &\leq& K_1\|f_1h_1(v_1+1)\|_{\infty}+K_2\|v_1\|_{2}\nonumber\\
					& \leq & K_1+K_2=C~\text{(say)}.
				\end{eqnarray}
				Thus we have
				\begin{eqnarray}
					-\Delta v_1&=&f_1h_1(v_1+1)\nonumber\\& \geq &f_1h_1(\|v_1\|_{\infty}+1)\nonumber\\
					&\geq &f_1h_1(C+1)\nonumber\\
					&>&0.
				\end{eqnarray}
				Since $f_1h_1(C+1)$ is identically not equal to zero, hence by the strong maximum principle over $(-\Delta)$ we have $v_1>0$. Thus, by our choice of $K\subset\subset\Omega$, there exists a constant $C_K$ such that $v_1(x)\geq C_K>0$ for almost every $x\in K$.\\
				Coming back to the proof of the lemma, we take the difference between the weak formulations of $\eqref{eq1}$ and $\eqref{eq2}$ respectively with the choice of test function being $(u_n-v_n)^-$.
				%\begin{equation}\label{eqq}
				%\int_{\Omega}(\nabla u_n-\nabla v_n)\nabla\varphi=\int_{\Omega} \bigg(h_n\left(u_n +\frac{1}{n}\right)-h_n\left(v_n +\frac{1}{n}\right) \bigg)f_n(x)\varphi + \int_\Omega \mu_n\varphi,\hspace{0.4cm}\forall\varphi\in W_0^{1}(\Omega).
				%\end{equation}
				It is easy to show that $u_n\geq v_n$ alomost everywhere in $\Omega$. For if not, i.e. if $u_n<v_n$ in $\Omega$, then
				\begin{align*}
					-\int_{\Omega} |\nabla(u_n-v_n)^-|^2 & =\int _{\Omega} \nabla (u_n-v_n).\nabla (u_n-v_n)^-
					\\& =\int _{\Omega} \bigg(h_n\left(u_n +\frac{1}{n}\right)-h_n\left(v_n +\frac{1}{n}\right) \bigg)f_n\cdot (u_n-v_n)^-\\
					& + \int_{\Omega} \mu_n\cdot(u_n-v_n)^- \\& \geq 0.
				\end{align*}
				This implies $u_n\geq v_n$ almost everywhere in $\Omega$ and hence in $K$.\\
				Thus we have showed that for every $K\subset\subset\Omega$ there exists $C_K$ such that $u_n\geq v_n\geq C_K>0$ almost everywhere in $K$.
			\end{proof}
		\end{lemma}
		\noindent Using the results proved so far we will now prove the existence of a solution to the problem in $\eqref{eq0}$. In order to do this we divide the problem into the following two cases.
		\subsection{The case of $0<\gamma < 1$}\label{sub1}
		%We have a strong control on $h$ at infinity by $\eqref{eqqq}$. Hence solutions to $\eqref{eq0}$ can be made to have finite energy for any integrable given data. From the classical theory of measure data problems we have $(u_n)$ is bounded in $W_0^{1,q}(\Omega)$ for every $q < \frac{N}{N-1}.$
		In this subsection, we consider the problem in $\eqref{eq1}$ for the case of $0<\gamma<1$.
		\begin{lemma}\label{l1}
			Let $u_n$ be a solution to the problem ($\ref{eq1}$), where $h$ satisfies ($\ref{eqq1}$) and ($\ref{eqqq}$), with $0<\gamma<1$. Then $(u_n)$ is bounded in $W_0^{1,q}(\Omega)$ for every $q<\frac{N}{N-1}$.
			\begin{proof}
				We follow the arguments used in $\cite{Benilan}$ to prove this lemma. We will first prove that ($\nabla u_n$) is bounded in $M^\frac{N}{N-1}(\Omega)$.	For this, we take $\varphi =T_k(u_n)$ as a test function in the weak formulation of $\eqref{eq1}$ and get
				\begin{equation}\label{eq3}
					\int_\Omega |\nabla T_k(u_n)|^2   \leq \int_\Omega h_n\left(u_n +\frac{1}{n}\right)T_k(u_n)f_n + \int_\Omega T_k(u_n) \mu_n.
				\end{equation}
				Now, $\frac{T_k(u_n)}{(u_n+\frac{1}{n})^{\gamma}}\leq\frac{u_n}{(u_n+\frac{1}{n})^{\gamma}}=\frac{u_n^{\gamma}}{(u_n+\frac{1}{n})^{\gamma}u_n^{\gamma-1}}\leq u_n^{1-\gamma}$.\\
				Using $\eqref{eqq1}$ and $\eqref{eqqq}$ in the right hand side of $\eqref{eq3}$ we have,
				\begin{align}\label{m}
					{\int_\Omega h_n\left(u_n +\frac{1}{n}\right)f_n T_k(u_n)}
					& \leq{ C_1 \int _{(u_n +\frac{1}{n}< \underline{K})}\frac{f_n T_k(u_n)} {(u_n +\frac{1}{n})^\gamma}} + \max_{[\underline{K},\overline{K}]} h(s) \int_{(\underline{K}\leq (u_n+\frac{1}{n})\leq \overline{K})} f_n T_k(u_n) \nonumber \\& + C_2 \int _{(u_n+\frac{1}{n} > \overline{K})}\frac{f_n T_k(u_n)} {(u_n+\frac{1}{n})^\theta}\nonumber\\&\leq  C_1\underline{K}^{1-\gamma} \int _{(u_n+\frac{1}{n}< \underline{K})}f + k \max_{[\underline{K},\overline{K}]} h(s) \int_{(\underline{K}\leq (u_n+\frac{1}{n})\leq \overline{K})}f\nonumber\\&+ \frac{C_2 k}{\overline{K}^{\theta}} \int _{(u_n+\frac{1}{n} > \overline{K})} f\nonumber \\&\leq Ck
				\end{align}
				and 
				\begin{align}\label{n}
				\int_\Omega T_k(u_n) \mu_n& \leq k\|\mu_n\|_{L^1(\Omega)} \nonumber\\
				&\leq Ck. 
				\end{align}
				Using the inequalities $\eqref{m}$ and $\eqref{n}$ in $\eqref{eq3}$, we obtain
				\begin{equation}\label{e}
					%\int_\Omega |\nabla u_n|^2=
					\int_\Omega |\nabla T_k(u_n)|^2  \leq Ck.
				\end{equation}
				Consider
				\begin{align*}
					\{|\nabla u_n|\geq  t\} & = \{|\nabla u_n|\geq  t,u_n< k\} \cup \{|\nabla u_n| \geq t,u_n \geq  k\}
					\\& \subset \{|\nabla u_n|\geq  t,u_n <k\} \cup \{u_n \geq k\}\subset \Omega.
				\end{align*}
				Then using the subadditivity property of Lesbegue measure $m$ we have,
				\begin{equation}\label{ee1}
					m(	\{|\nabla u_n|\geq t\}) \leq m(\{|\nabla u_n|\geq t,u_n< k\}) + m(\{u_n \geq  k\}).
				\end{equation}
				Therefore, from the Sobolev inequality 
				\begin{align}\label{f}
				\Bigg(\int_\Omega |T_k(u_n)|^{2^*}\Bigg)^{\frac{2}{2^*}}&\leq\frac{1}{\lambda_1} \int_{\Omega}|\nabla T_k(u_n)|^2 \nonumber\\
				&\leq Ck,
				\end{align}
				where $\lambda_1$ is the first eigenvalue of the Laplacian operator. By restricting the left hand side of $\eqref{e}$ over $I_1=\{|\nabla u_n|\geq  t,u_n< k\}$, we get
				\begin{align}
					m(\{|\nabla u_n|\geq  t,u_n< k\})&\leq \frac{1}{t^2}\int_\Omega |\nabla T_k(u_n)|^2\nonumber\\
					&\leq \frac{Ck}{t^2}, ~\forall k>1. \nonumber
				\end{align}
				 Again by restricting the integral on the left hand side of $\eqref{f}$ over $I_2=\left\lbrace x\in\Omega:u_n\geq k \right\rbrace$, in which $T_k(u_n)=k$, we obtain $$k^2m(\{u_n\geq k\})^{\frac{2}{2^*}}\leq Ck.$$
				This implies $$m(\{u_n\geq k\})\leq \frac{C}{k^\frac{N}{N-2}},~ \forall k\geq1.$$
				Hence, $(u_n)$ is bounded in $M^{\frac{N}{N-2}}(\Omega)$. Now $\eqref{ee1}$ becomes
				\begin{align}
					m(	\{|\nabla u_n|\geq t\})& \leq m(\{|\nabla u_n|\geq t,u_n< k\}) + m(\{u_n \geq k\})\nonumber\\
					&\leq \frac{Ck}{t^2} + \frac{C}{k^\frac{N}{N-2}},~ \forall k>1.\nonumber
				\end{align}
		On choosing $k=t^{\frac{N-2}{N-1}}$, we get $$ m(\{|\nabla u_n|\geq t\})\leq \frac{C}{t^\frac{N}{N-1}},~ \forall t\geq 1.$$
	We have thus proved that $(\nabla u_n)$ is bounded in $M^{\frac{N}{N-1}}(\Omega)$. Therefore, the property in $\eqref{mar}$ implies that $(u_n)$ is bounded in $W_0^{1,q}$ for every $q<\frac{N}{N-1}$.			
	\end{proof}
		\end{lemma}
		\begin{theorem}\label{t1}
			There exists a weak solution $u$ of $\eqref{eq0}$ in $W_0^{1,q}(\Omega)$ for every $q<\frac{N}{N-1}$.
		\end{theorem}
			\begin{proof}	
				With the assumptions made in Lemma $\ref{l1}$, there exists $u$ such that the sequence ($u_n$) converges weakly to $u$ in $W_0^{1,q}(\Omega)$ for every $q<\frac{N}{N-1}$. This implies that for $\varphi$ in $C_c^1(\bar{\Omega})$
				$$\lim_{n\rightarrow \infty} \int_{\Omega} \nabla u_n . \nabla\varphi = \int_{\Omega}\nabla u .\nabla\varphi.$$
				In addition to this, by the compact embedding we conclude that $u_n$ converges to $u$ strongly in $L^1(\Omega)$ and hence pointwise upto a subsequence almost everywhere in $\Omega$. Thus, for $\varphi$ belonging to $C_c^1(\bar{\Omega})$, we have
				\begin{align*}
					0 & \leq |h_n\left(u_n+\frac{1}{n}\right) f_n\varphi|\\& \leq	{\left\{\begin{array}{ll}
							\frac{C_1\parallel \varphi \parallel_{L^\infty(\Omega)}f}{C_K^\gamma}, & \text{if \space}   u_n+\frac{1}{n}< \underline{K} \\
							{M\parallel \varphi \parallel_{L^\infty(\Omega)}f}, & \text{if \space}   \underline{K}\leq u_n+\frac{1}{n}\leq \overline{K}~\\
							\frac{C_2\parallel \varphi \parallel_{L^\infty(\Omega)}f}{C_K^\theta}, & \text{if \space}   u_n+\frac{1}{n}> \overline{K}
						\end{array}
						\right. }
				\end{align*}
				where, $M>0$ and $K$ is the set $\{x\in\Omega:\varphi(x)\neq0\}$. On applying the dominated convergence theorem we get
				$$\lim_{n\rightarrow \infty} \int_{\Omega}  h_n\left(u_n+\frac{1}{n}\right) f_n\varphi = \int_{\Omega}h( u) f\varphi.$$
				Hence, on passing the limit $n\rightarrow\infty$ in the last term of $\eqref{weak}$ involving $\mu_n$, we obtain a weak solution $u$ of $\eqref{eq0}$ in $W_0^{1,q}(\Omega)$ for every $q<\frac{N}{N-1}$. This completes the proof.
			\end{proof}
		
		\subsection{The case of $\gamma \geq 1$}\label{sub2}
		Since this is a strongly singular case, we can obtain local estimates on $u_n$ in the Sobolev space. We will globally estimate $ \left(T_k^{\frac{\gamma+1}{2}}(u_n)\right)$ in $W_0^{1,2}(\Omega)$ with the aim of giving a sense to the boundary values of $u$ at least in a weaker sense when compared to the trace sence.
		\begin{lemma}\label{l2}
			Let $u_n$ be a solution of $\eqref{eq1}$ with $\gamma \geq 1$. Then $\left(T_k^{\frac{\gamma+1}{2}}(u_n)\right)$ is bounded in $W_0^{1,2}(\Omega)$ for every fixed $k>0$.
			\begin{proof}
				Consider $\varphi=T_k^\gamma(u_n)$ as a test function in $\eqref{eq1}$. We have
				\begin{equation}\label{eq4}
					\gamma\int_\Omega \nabla u_n.\nabla T_k(u_n)T_k^{\gamma-1}(u_n) =\int_{\Omega} h_n\left(u_n+\frac{1}{n}\right)f_nT_k^\gamma(u_n) + \int_{\Omega}T_k^\gamma(u_n)\mu_n.
				\end{equation}
				Since, $\gamma\geq 1$ and by the definition of $T_k(u_n)$, we estimate the term on the left hand side of $\eqref{eq4}$ as
				\begin{equation}\label{g}
				\gamma\int_\Omega \nabla u_n.\nabla T_k(u_n)T_k^{\gamma-1}(u_n)\geq\gamma\int_\Omega |\nabla T_k^{\frac{\gamma+1}{2}}(u_n)|^2.
				\end{equation}
				Recal that $\frac{T_k^\gamma(u_n)}{(u_n+\frac{1}{n})^\gamma}\leq \frac{u_n^\gamma}{(u_n+\frac{1}{n})^\gamma}\leq 1$, then the term on the right hand side of $\eqref{eq4}$ can be estimated as
				\begin{align}\label{h}
					\int_{\Omega} h_n\left(u_n+\frac{1}{n}\right)f_nT_k^\gamma(u_n) + \int_{\Omega}T_k^\gamma(u_n)\mu_n
					& \leq{ C_1 \int _{(u_n +\frac{1}{n}< \underline{K})}\frac{f_n T_k^\gamma(u_n)} {(u_n +\frac{1}{n})^\gamma}} + C_2 \int _{(u_n+\frac{1}{n} > \overline{K})}\frac{f_n T_k^\gamma(u_n)} {(u_n+\frac{1}{n})^\theta} \nonumber  \\& + \max_{[\underline{K},\overline{K}]} h(s) \int_{(\underline{K}\leq (u_n+\frac{1}{n})\leq \overline{K})} f_n T_k^\gamma(u_n) + k^\gamma \int_{\Omega} \mu_n\nonumber
					\\&  \leq  C_1 \int _{(u_n+\frac{1}{n}< \underline{K})}f +  \frac{C_2 k^\gamma}{\overline{K}^{\theta}} \int _{(u_n+\frac{1}{n} > \overline{K})} f \nonumber \\& +  k^\gamma \max_{[\underline{K},\overline{K}]} h(s) \int_{(\underline{K}\leq (u_n+\frac{1}{n})\leq \overline{K})}f + k^\gamma \int_\Omega \mu_n\nonumber
					\\& \leq C(k,\gamma)k^\gamma.
				\end{align}
				On combining the inequalities in $\eqref{g}$ and $\eqref{h}$, we get
				\begin{equation}\label{k}
					\int_{\Omega} |\nabla T_k^{\frac{\gamma+1}{2}}(u_n)|^2 \leq Ck^\gamma.
				\end{equation}
				Therefore,  $\left(T_k^{\frac{\gamma+1}{2}}(u_n)\right)$ is bounded in  $W_0^{1,2}(\Omega)$ for every fixed $k>0$.
			\end{proof}
		\end{lemma}
		\noindent In order to pass the limit $n\rightarrow\infty$ in the weak formulation  $\eqref{weak}$, we require local estimates on $(u_n)$. We prove the following lemma.
		\begin{lemma}\label{l3}
			Let $u_n$ be a solution of $\eqref{eq1}$ with $\gamma\geq1$. Then ($u_n$) is bounded in $W_{loc}^{1,q}(\Omega)$ for every $q<\frac{N}{N-1}$.
			\begin{proof}
				We prove this theorem in two steps.\\
				$\boldmath{\text{Step 1.}}$ We claim that $\left(G_1(u_n)\right)$ is bounded in  $W_0^{1,q}(\Omega)$ for every $q<\frac{N}{N-1}$.\\
				It is apparent that $G_1(u_n)=0$, when $0\leq u_n\leq 1$ and $G_1(u_n)=u_n-1$, when $u_n>1$. So $\nabla G_1(u_n)=\nabla u_n$ for $u_n>1$.\\
				Now, we need to show that $\left(\nabla G_1(u_n)\right)$ is bounded in the Marcinkiewicz space $M^{\frac{N}{N-1}}(\Omega)$. We observe
				\begin{align*}
					\{|\nabla u_n|> t, u_n>1\} & = \{|\nabla u_n|> t,1<u_n\leq k+1\} \cup \{|\nabla u_n| > t,u_n > k+1\}
					\\& \subset \{|\nabla u_n|> t,1<u_n\leq k+1\} \cup \{u_n > k+1\}\subset \Omega.
				\end{align*}
				Hence, by the subadditivity of Lebesgue measure $m$, we have
				\begin{equation}\label{eq5}
					m(	\{|\nabla u_n|> t,u_n>1\}) \leq m(\{|\nabla u_n|> t,1<u_n\leq k+1\}) + m(\{u_n > k+1\}).
				\end{equation}
				In order to estimate $\eqref{eq5}$ we take $\varphi=T_k(G_1(u_n))$, for $k>1$, as a test function in $\eqref{eq1}$. We observe that $\nabla T_k(G_1(u_n))= \nabla u_n$ only when $1<u_n \leq k+1$, otherwise it is equal to zero and  $T_k(G_1(u_n))=0$ when $ u_n\leq 1$. Thus we have
				\begin{align}\label{i}
					\int_\Omega |\nabla T_k(G_1(u_n))|^2 & \leq \int_\Omega h_n\left(u_n+\frac{1}{n}\right)f_nT_k(G_1(u_n)) + \int_{\Omega}T_k(G_1(u_n))\mu_n\nonumber
					\\& \leq{ C_1 \int _{(u_n +\frac{1}{n}< \underline{K})}\frac{f_n T_k(G_1(u_n))} {(u_n +\frac{1}{n})^\gamma}} + \max_{[\underline{K},\overline{K}]} h(s) \int_{(\underline{K}\leq (u_n+\frac{1}{n})\leq \overline{K})} f_n T_k(G_1(u_n)) \nonumber \\& + C_2 \int _{(u_n+\frac{1}{n} > \overline{K})}\frac{f_n T_k(G_1(u_n))} {(u_n+\frac{1}{n})^\theta}+k\int_\Omega\mu_n\nonumber
					\\&  \leq { C_1k \int _{(u_n+\frac{1}{n}< \underline{K})}\frac{f_n}{(1+\frac{1}{n})^\gamma} + k \max_{[\underline{K},\overline{K}]} h(s) \int_{(\underline{K}\leq (u_n+\frac{1}{n})\leq \overline{K})}f_n}\nonumber \\& + \frac{C_2 k}{\overline{K}^{\theta}} \int _{(u_n+\frac{1}{n} >
						\overline{K})} f_n +k\int_\Omega\mu_n\nonumber
					\\& \leq Ck.
				\end{align}
				By restricting the integral in $\eqref{i}$ over $J_1={\left\lbrace 1<u_n\leq k+1 \right\rbrace}$, we get
					\begin{align}
					\int_{\left\lbrace 1<u_n\leq k+1 \right\rbrace} |\nabla T_k(G_1(u_n))|^2 \nonumber& = \int_{\left\lbrace 1<u_n\leq k+1 \right\rbrace} |\nabla u_n|^2\nonumber  \\& \geq \int_{\left\lbrace |\nabla u_n|>t, 1<u_n\leq k+1 \right\rbrace} |\nabla u_n|^2\nonumber \\&\geq t^2m(\{|\nabla u_n|> t,1<u_n\leq k+1\}).\nonumber
				\end{align}
				Thus, $$m(\{|\nabla u_n|> t,1<u_n\leq k+1\})\leq \frac{Ck}{t^2}, ~\forall k \geq 1.$$
				According to $\eqref{k}$ in the proof of Lemma $\ref{l2}$, one can see that
				$$\int_{\Omega} |\nabla T_k^{\frac{\gamma+1}{2}}(u_n)|^2 \leq Ck^\gamma,~ \forall k>1.$$
				Therefore, from the Sobolev inequality \begin{align}\label{j}
				\Bigg(\int_\Omega |T_k^{\frac{\gamma+1}{2}}(u_n)|^{2^*}\Bigg)^{\frac{2}{2^*}}&\leq \frac{1}{\lambda_1}\int_{\Omega}|\nabla T_k^{\frac{\gamma+1}{2}}(u_n)|^2\nonumber\\ &\leq Ck^{\gamma},
				\end{align}
				where $\lambda_1$ is the first eigenvalue of the Laplacian operator.
			By restricting the integral on the left hand side of $\eqref{j}$ over $J_2=\{x: u_n(x)> k+1\}$, we obtain $$k^{\gamma+1}m(\{u_n>k+1\})^{\frac{2}{2^*}}\leq Ck^{\gamma}$$
				so that $$m(\{u_n>k+1\})\leq \frac{C}{k^\frac{N}{N-2}},~ \forall k\geq1.$$
				So, $(u_n)$ is bounded in $M^{\frac{N}{N-2}}(\Omega)$, i.e. $(G_1(u_n))$ is also bounded in $M^{\frac{N}{N-2}}(\Omega)$.\\
				Now from $\eqref{eq5}$, we have
				\begin{eqnarray}m(\{|\nabla u_n|> t,u_n>1\}) &\leq & m(\{|\nabla u_n|> t,1<u_n\leq k+1\}) + m(\{u_n > k+1\})\nonumber\\&\leq& \frac{Ck}{t^2} + \frac{C}{k^\frac{N}{N-2}}, \forall k>1.\nonumber
				\end{eqnarray}
				On choosing $k=t^{\frac{N-2}{N-1}}$ we get  $$ m(\{|\nabla u_n|> t,u_n>1\})\leq \frac{C}{t^\frac{N}{N-1}},~ \forall t\geq 1.$$
				We thus proved that $(\nabla u_n)=(\nabla G_1(u_n))$ is bounded in $M^{\frac{N}{N-2}}(\Omega)$. Hence, by the property in ($\ref{mar}$), we conclude that $(G_1(u_n))$ is bounded in $W_0^{1,q}(\Omega)$ for every $q<\frac{N}{N-1}$.\\			
				$\boldmath{\text{Step 2.}}$ We claim that $(T_1(u_n))$ is bounded in $W_{loc}^{1,2}(\Omega)$.\\
			To prove this claim	we need to examine the behaviour, for small values, of $u_n$ for each $n$. For this we first prove that for every $K\subset\subset \Omega$,
				\begin{equation}\label{eq7}
					\int_K |\nabla T_1(u_n)|^2 \leq C.
				\end{equation}
				We have already proved in Lemma $\ref{l4}$ that $u_n\geq C_K>0$ on $K\subset\subset\Omega$. On using $\varphi=T_1^\gamma(u_n)$ as a test function in $\eqref{weak}$, we get
				\begin{align}\label{eq8}
					\int_\Omega \nabla u_n . \nabla T_1(u_n) T_1^{\gamma-1}(u_n)&= \int_\Omega h_n\left(u_n+\frac{1}{n}\right)f_nT_1^{\gamma}(u_n) +\int_\Omega T_1^{\gamma}(u_n)\mu_n\nonumber\\
					& \leq C.
				\end{align}
			We observe that
				\begin{align}\label{eq9}
					\int_\Omega \nabla u_n . \nabla T_1(u_n) T_1^{\gamma-1}(u_n)  & \geq\int_K |\nabla T_1(u_n)|^2 T_1^{\gamma-1}(u_n)\nonumber\\
					& \geq  C_K^{\gamma-1}\int_{K}|\nabla T_1(u_n)|^2.
				\end{align}
			Inequalities $\eqref{eq8}$ and $\eqref{eq9}$ together yields $\eqref{eq7}$. \\
			Since $u_n=T_1(u_n)+G_1(u_n)$, we conclude that ($u_n$) is bounded in $W_{loc}^{1,q}(\Omega)$ for every $q<\frac{N}{N-1}$.
			\end{proof}
		\end{lemma}	
		\noindent We now state and prove the existence result.
		\begin{theorem}
			Let $\gamma\geq1$. Then there exists a weak solution $u$ of $\eqref{eq0}$ in $W_{loc}^{1,q}(\Omega)$ for every $q<\frac{N}{N-1}$.
			\begin{proof}
				The proof of this theorem is a straightforward application of the results in Theorem $\ref{t1}$, Lemma $\ref{l2}$ and Lemma $\ref{l3}$.
			\end{proof}
		\end{theorem}
		\section{Further discussion of the case $0<\gamma<1$.}
		In this section we will consider $\Omega$ that has a boundary $\partial\Omega$ of class $C^{2,\beta}$ for some $0<\beta<1$. We consider the following semilinear elliptic problem
		\begin{eqnarray}\label{e1}	
			-\Delta u&=& h(u)f+\mu ~\text{in}~\Omega,\nonumber\\
			u &=& 0 ~\text{on}~ \partial\Omega,	
		\end{eqnarray}
		where, $0<\gamma<1$, $f\in C^{\beta} (\bar{\Omega})$ such that $f>0$ in $\bar{\Omega}$ and $\mu$ is a nonnegative, bounded, Radon measure on $\Omega$. We will show the existence of a nonnegative very weak solution to the problem $\eqref{e1}$. Before proving this we give a few definitions.
		\begin{definition}\label{d1}
			A  very weak solution to problem $\eqref{e1}$ is a function $u\in L^1(\Omega)$ such that $u>0$ a.e. in $\Omega$, $fh(u) \in L^1(\Omega) $ and
			\begin{equation}\label{e2}
				-\int_{\Omega}u\Delta \varphi=\int_{\Omega}h(u)f\varphi +\int_{\Omega}\varphi d\mu,~ \forall\varphi\in C_0^2({\bar{\Omega}}).
			\end{equation}
		\end{definition}
		\begin{definition}\label{d2}
			A function $\underline{u}$ is a subsolution for $\eqref{e1}$ if $\underline{u}\in L^1(\Omega)$,  $\underline{u}>0$ in $\Omega$, $fh(\underline{u})\in L^1(\Omega)$
			and
			\begin{equation}\label{e3}
				-\int_{\Omega}\underline{u}\Delta \varphi\leq\int_{\Omega}h(\underline{u})f\varphi +\int_{\Omega}\varphi d\mu,~ \forall\varphi\in C_0^2({\bar{\Omega}}),~ \varphi\geq0.
			\end{equation}
			Equivalently, $\bar{u}$ is said to be a supersolution for the problem $\eqref{e1}$ if $\bar{u}\in L^1(\Omega)$,    $\bar{u}>0$ in $\Omega$, $fh(\bar{u})\in L^1(\Omega)$ and
			\begin{equation}\label{e6}
				-\int_{\Omega}\bar{u}\Delta \varphi\geq\int_{\Omega}h(\bar{u})f\varphi +\int_{\Omega}\varphi d\mu,~ \forall\varphi\in C_0^2({\bar{\Omega}}),~\varphi\geq0.
			\end{equation}
		\end{definition}
	\noindent We now prove the following two theorems in order to guarantee the existence of a nonnegative solution to $\eqref{e1}$ in the sense of Definition $\ref{d1}$.
		\begin{theorem}\label{t2}
			Let $\underline{u}$ be a subsolution  and $\bar{u}$ be a supersolution to the problem in $\eqref{e1}$ with $\underline{u}\leq\bar{u}$ in
			$\Omega$, then there exists a solution $u$ to $\eqref{e1}$ in the sense of Definition $\ref{d1}$ such that $\underline{u}\leq u\leq \bar{u}$.
			\begin{proof}
				We will follow the arguments due to Montenegro \& Ponce $\cite{Ponce}$.
				%First of all we will try to modify the singular nonlinearity. The aim is to show that the solution of this
				%	truncated problem works out as solution to the problem $\eqref{e1}$.
				We define $\bar{g}:\Omega \times \mathbb{R}\rightarrow \mathbb{R}$ as
				$$\bar{g}(x,t)=	\left\{
				\begin{array}{ll}
				f(x)h(\underline{u}(x)) & \text{if}~ t<\underline{u}(x),\\
				f(x)h(t) & \text{if}~ \underline{u}(x)\leq t\leq \bar{u}(x),\\
				f(x)h(\bar{u}(x)) & \text{if}~ t>\bar{u}(x).
				\end{array}
				\right. $$
				Moreover, $\underline{u}>0$ and hence $\bar{g}$ is well defined a.e. in $\Omega$. For each fixed $v\in L^1(\Omega)$ we have that $\bar{g}(x,v(x))\in L^1(\Omega)$. We divide the proof into two steps.\\
				$\boldmath{\text{Step 1.}}$ We claim that if $u$ satisfies
				\begin{eqnarray}\label{e4}	
					-\Delta u&=& \bar{g}(x,u)+\mu ~\text{in}~\Omega,\nonumber\\
					u &=& 0 ~\text{on}~ \partial\Omega,	
				\end{eqnarray}
				then $\underline{u}\leq u\leq \bar{u}$. Thus $\bar{g}(.,u)=fh(u)\in L^1(\Omega)$ and $u$ is a solution to $\eqref{e1}$.\\
				The very weak formulation of $\eqref{e4}$ is given by
				\begin{equation}\label{e5}
					-\int_{\Omega}u\Delta \varphi= \int_{\Omega}\bar{g}(x,u)\varphi +\int_\Omega\varphi d\mu,~\forall\varphi\in C_0^2(\bar\Omega).
				\end{equation}
				We need to prove that $u\leq\bar{u}$ in $\Omega$. The proof of the other side of the inequality, $\underline{u}\leq u$, follows similarly.\\
				We will show that $u$ is a solution to $\eqref{e4}$ and $\bar{u}$ is a supersolution to $\eqref{e1}$. Subtracting equation $\eqref{e5}$ from $\eqref{e6}$ we have, for every $\varphi \in C_0^2(\bar{\Omega})$ such that $\varphi \geq0$,
				\begin{align}
				-\int_{\Omega}(u-\bar{u})\Delta \varphi&\leq \int_{\Omega}\left(\bar{g}(x,u)-fh(\bar{u})\right)\varphi \nonumber\\
				&=\int_{\Omega}\chi_{\{u\leq \bar{u}\}}\left(\bar{g}(x,u)-fh(\bar{u})\right)\varphi.\nonumber
				\end{align}
				On applying the Kato type inequality from the Appendix, we get
				\begin{align}
				\int_{\Omega}(u-\bar{u})^+& \leq \int_{\Omega}\chi_{\{u\leq \bar{u}\}} \left(\bar{g}(x,u)-fh(\bar{u})\right) (sign_+(u-\bar{u}))\varphi\nonumber\\
				&=0.\nonumber
				\end{align}
				This further implies that
				$$\int_{\Omega}(u-\bar{u})^+\leq 0.$$
				Thus $u\leq\bar{u}$ a.e. in $\Omega$. Similarly, one can show that $\underline{u}\leq u$ a.e. in $\Omega$. Thus we have proved that $\underline{u}\leq u\leq\bar{u}$ a.e. in $\Omega$.\\
				$\boldmath{\text{Step 2.}}$ Now we prove that a solution to the problem in $\eqref{e4}$ exists. Let us define
				$$G:L^1(\Omega)\rightarrow L^1(\Omega).$$ 
				The map $G$ is so defined that it assigns to every $v\in L^1(\Omega)$ a solution $u$ to the following linear problem
				\begin{eqnarray}	
					-\Delta u&=& \bar{g}(x,v)+\mu ~\text{in}~\Omega,\nonumber\\
					u &=& 0 ~\text{on}~ \partial\Omega.\label{linprob}	
				\end{eqnarray}
				The problem in ($\ref{linprob}$) admits a unique solution for a given Radon measure due to \cite{Stampacchia}.
				We need to show that this map is continuous in $L^1(\Omega)$. Let us choose a sequence $(v_n)$ converging
				to some function $v$ in $L^1(\Omega)$. Then by the definition of $\bar{g}$ and $h$ being a nonincreasing, continuous function we get
				$$|\bar{g}(x,v_n(x))|\leq fh(\underline{u}).$$
				Hence, using the dominated convergence theorem, we conclude that
				$$\|\bar{g}(x, v_n)-\bar{g}(x, v)\|_{L^1(\Omega)}\rightarrow 0.$$ By $\cite{bhakta}$, the linear problem $(\ref{linprob})$ has a unique very weak solution corresponding to this $v$. Thus
				\begin{eqnarray}
					\lim_{n\rightarrow\infty}\left(-\int_{\Omega}u_n\Delta\varphi\right) &=& \lim_{n\rightarrow\infty}\int_{\Omega}fh(v_n)\varphi+\int_{\Omega}\varphi d\mu\nonumber\\
					&=&\int_{\Omega}fh(v)\varphi+\int_{\Omega}\varphi d\mu\nonumber\\
					&=& -\int_{\Omega}u\Delta\varphi.\nonumber
				\end{eqnarray}
				Hence,  $u = G(v)$. It can be seen from Th\'{e}or\`{e}me 9.1 in \cite{Stampacchia} that
				\begin{align}
				\|u_n-u\|_{L^1(\Omega)}&\leq\|u_n-u\|_{W_0^{1,q}(\Omega)}\nonumber\\
				&\leq \|(\bar{g}(x,v_n)+\mu)-(\bar{g}(x,v)+\mu)\|_{\mathcal{M}(\Omega)}\nonumber\\
				&=\|\bar{g}(x,v_n)-\bar{g}(x,v)\|_{L^1(\Omega)}\rightarrow 0, ~\text{(as)}~n\rightarrow\infty. \nonumber
				\end{align} 
				Hence, $\|u_n-u\|_{L^1(\Omega)}=\|G(v_n)-G(v)\|_{L^1(\Omega)}\rightarrow 0$ as $n\rightarrow\infty.$
				 Therefore, $G$ is continuous.\\
				We are still left to prove that the set $G(L^1(\Omega))$ is bounded and relatively compact in $L^1(\Omega)$. For every $v\in L^1(\Omega)$ we have
				\begin{align}
					\parallel\bar{g}(x,v)+\mu\parallel_{\mathcal{M}(\Omega)}& \leq \parallel\bar{g}(x,v)\parallel_{\mathcal{M}(\Omega)}+\parallel \mu\parallel_{\mathcal{M}(\Omega)}\nonumber\\
					&\leq ~\parallel f h(\underline{u})\parallel_{L^1(\Omega)}+\parallel\mu\parallel_{\mathcal{M}(\Omega)}.\nonumber
				\end{align}
				Again, by Th\'{e}or\`{e}me 9.1 in $\cite{Stampacchia}$, we see that $G(v)$ is bounded in $W_0^{1,q}(\Omega)$ for every $q<\frac{N}{N-1}$. Therefore, by the Rellich-Kondrachov theorem  we get $G(L^1(\Omega))$ is bounded and hence relatively compact in $L^1(\Omega)$.\\
			We now apply the Schauder fixed point theorem  to see that $G$ has a fixed point $u\in L^1(\Omega)$. This fixed point of $G$ is a very weak solution to the problem $\eqref{e1}$. Also by step 1 we have that this solution $u$ satisfies $\underline{u}\leq u\leq \bar{u}$ a.e. in $\Omega$.
			\end{proof}
		\end{theorem}
		\begin{theorem}\label{t3}
			There exists a nonnegative solution to the problem $\eqref{e1}$ in the sense of Definition $\ref{d1}$.
			\begin{proof}
				We will apply the Theorem $\ref{t2}$ for which we find a subsolution and a supersolution
				to the problem $\eqref{e1}$ in the sense of Definition $\ref{d2}$. We first find a subsolution. Let us consider the problem
				\begin{eqnarray}\label{e7}	
					-\Delta v&=& h(v)f ~\text{in}~\Omega,\nonumber\\
					%v&>&0,~\text{in}~\Omega\nonumber\\
					v &=& 0 ~\text{on}~ \partial\Omega.
				\end{eqnarray}
				%It is proved in $\cite{lazer}$ that there  exists a solution $v\in L^1(\Omega)$ to the problem $\eqref{e7}$ such that, if we denote  $\varphi_1$ as the first eigenfunction of the Laplacian with $\varphi= 0$ on $\partial\Omega$, then we have $$v\geq\epsilon\varphi_1\text{, for a small positive}\hspace{0.2cm}\epsilon.$$ Therefore $$h(v)\leq c(\epsilon)h(\varphi_1).$$
				%			Now, due to the assumption that $\partial\Omega$ is of class $C^{2,\beta}$ we can see that $h(\varphi_1)\in L^1(\Omega)$ if and only if $\gamma<1$ by using the Lemma, Page 726 in $\cite{lazer}$.
				The existence of a very weak solution in $L^1(\Omega)$ to the problem in ($\ref{e7}$) can be proved by the arguments used in the Theorem $\ref{t2}$, i.e. by using the Schauder fixed point theorem. Consider the eigenfunction $\phi_1>0$ of ($-\Delta$) corresponding to the smallest eigenvalue $\lambda_1$ with $\phi_1|_{\partial\Omega}=0$ \cite{Evans}. Observe that \begin{eqnarray}\label{positivity}
					-\Delta(\epsilon\phi_1)-h(\epsilon\phi_1)f&<&0 \nonumber\\
					&=&-\Delta v-h(v)f \nonumber
				\end{eqnarray}
				since (i) $\phi_1>0$ and a choice of sufficiently small $\epsilon>0$, (ii) the nonincreasing nature of $h$ and (iii) $v$ is a solution to ($\ref{e7}$). Hence, we have  $v > 0$ in $\Omega$.
				%, $h(v)f \in L^1(\Omega)$, and
				%			$$-\int_\Omega v\Delta \varphi=\int_{\Omega}h(v)f\varphi, \hspace{0.3cm} \forall\varphi\in C_0^2(\bar{\Omega}).$$
				Since $\mu$ is a nonnegative, bounded, Radon measure we get the following inequality.
				$$-\int_\Omega v\Delta \varphi\leq\int_{\Omega}h(v)f\varphi +\int_{\Omega}\varphi d\mu,~ \forall\varphi\in C_0^2(\bar{\Omega}), ~\varphi\geq0$$
				and hence $v$ is a subsolution to the problem $\eqref{e1}$. We now look for a supersolution to the problem in $\eqref{e1}$. Let $w$ be the solution of
				\begin{eqnarray}\label{e8}	
					-\Delta w&=& \mu ~\text{in}~\Omega,\nonumber\\
					w &=& 0 ~\text{on}~ \partial\Omega.	
				\end{eqnarray}
				Since $\mu\geq 0$, by the maximum principle on Laplacian we have $w\geq 0$.
				%		There exists a positive solution to the problem $\eqref{e8}$ in classical sense $\cite{Stampacchia}$.
				Let us denote $z=w + v$, where $v$ is a solution to ($\ref{e7}$). Thus
				\begin{align}
				-\int_\Omega z\Delta \varphi&=-\int_\Omega w\Delta \varphi -\int_\Omega v\Delta \varphi\nonumber\\
				&=\int_{\Omega}h(v)f\varphi +\int_{\Omega}\varphi d\mu,~ \forall\varphi\in C_0^2(\bar{\Omega}).\nonumber
				\end{align}
				We know that $w$ is nonnegative and hence we have $0<h(z)\leq h(v)$. Therefore,
				$$\int_{\Omega}h(z)f\varphi +\int_{\Omega}\varphi d\mu\leq \int_{\Omega}h(v)f\varphi +\int_{\Omega}\varphi d\mu,~ \forall\varphi\in C_0^2(\bar{\Omega}),~ \varphi\geq0,$$
				i.e., $z$ is a positive function in $L^1(\Omega)$ such that $h(z)\leq h(v)\in L^1(\Omega)$ and
				$$-\int_\Omega z\Delta \varphi\geq\int_{\Omega}h(z)f\varphi+\int_{\Omega}\varphi d\mu,~ \forall\varphi\in C_0^2(\bar{\Omega}),~ \varphi\geq0.$$
				Therefore, $z$ is a supersolution to $\eqref{e1}$. We can now apply Theorem $\ref{t2}$ to conclude that there exists a  solution $u$ to problem $\eqref{e1}$ in the sense of Definition $\ref{d1}$.				
			\end{proof}
		\end{theorem}
		\subsection{Relaxation of assumption on $f$}
 Theorem $\ref{t3}$ has been proved by assuming a strong regularity on $f$, i.e. $f$ belongs to $C^{\beta}(\bar{\Omega})$ for some $0 <\beta< 1$. In this subsection we relax our assumption on $f$ in order to prove the existence of solution in the sense of Definition $\ref{d1}$.\\
		For a fixed $\delta>0$, let us define  $\Omega_\delta=\{x\in\Omega :\text{dist}(x,\partial\Omega)<\delta \}$ and let $f$ be an almost everywhere positive function in $L^1(\Omega)\cap L^\infty(\Omega_\delta)$.
		\begin{theorem}
			Let $f\in L^1(\Omega)\cap L^{\infty}(\Omega_\delta)$ such that $f>0$ a.e. in $\Omega$ for some fixed $\delta>0$. Then there exists a solution to the problem ($\ref{e1}$) in the sense of Definition $\ref{d1}$.
			\begin{proof}
				We consider the following sequence of problems
				\begin{eqnarray}	
					-\Delta v_n&=& h_n\left(v_n+\frac{1}{n}\right)f_n ~\text{in}~\Omega,\nonumber\\
					v_n &=& 0 ~\text{on}~ \partial\Omega.	
				\end{eqnarray}
				In Lemma $\ref{l4}$ we have proved that the nondecreasing sequence $(v_n)$ converges to a solution of the problem in ($\ref{e7}$) and for each fixed $n$, the function $v_n$ belongs to $L^{\infty}(\Omega)$. So the function $h_1(v_1+1)f_1$ also belongs to $L^{\infty}(\Omega)$. We now can apply the Lemma 3.2 in $\cite{Brezis}$ so as to obtain
				\begin{align}
				\frac{v_1(x)}{d(x)}&\geq C\int_\Omega d(y)f_1(y) h_1\left(\parallel v_1\parallel_{L^{\infty}(\Omega)}+1\right)dy \nonumber\\
				&\geq C>0\nonumber
				\end{align}
			 for every $x$ in $\Omega$,
				where $d(x) = d(x, \partial\Omega)$. Thus, we have $v(x)\geq v_1(x)\geq Cd(x)$, a.e. on $\Omega$. Therefore, as $f\in L^{\infty}(\Omega_\delta)$, we have $h(v)f\in L^1(\Omega)$ because (i) $h(v)f\leq h(Cd(x))f$ and (ii) $h(Cd(x))f$ is integrable for every $\gamma<1$. Hence, the subsolution is bounded from below. Proceeding further by using the arguments used in the proof of the Theorem $\ref{t3}$, one can produce a super solution to $\eqref{e1}$. Now using the result proved in Theorem $\ref{t2}$, we conclude the existence of a solution to the problem in $\eqref{e1}$ in the sense of Definition $\ref{d1}$.			
			\end{proof}
		\end{theorem}
		\section{Appendix}
		We prove the Kato type inequality for the problem
		\begin{eqnarray}\label{eqq0}	
			-\Delta u&=& h(u)f+\mu ~\text{in}~\Omega,\nonumber\\
			u &=& 0 ~\text{on}~ \partial\Omega,\\
			u &>& 0 ~\text{in}~ \Omega,\nonumber	
		\end{eqnarray}
		where $f>0$ and $u\in L^1(\Omega)$ is a very weak solution with $u>0$ a.e. in $\Omega$ and $fh(u)\in L^1(\Omega)$ .\\
		Let $u_1$ and $u_2$ are two very weak solutions to the problem $\eqref{eqq0}$ with measure sources $\mu_1$ and $\mu_2$, respectively. Hence, $u_1,u_2\in L^1(\Omega)$ and $h(u_1)f,h(u_2)f\in L^1(\Omega)$. Then for every $\varphi\in C_0^2(\bar{\Omega})$, the very weak formulations corresponding to the problem $\eqref{eqq0}$ are
		$$-\int_{\Omega}u_1\Delta\varphi = \int_{\Omega}h(u_1)f\varphi +\int_{\Omega}\varphi d\mu_1$$ and $$-\int_{\Omega}u_2\Delta\varphi = \int_{\Omega}h(u_2)f\varphi +\int_{\Omega}\varphi d\mu_2.$$
		On taking the difference between the two formulations we get
		$$-\int_{\Omega}(u_1-u_2)\Delta\varphi = \int_{\Omega}f(h(u_1)-h(u_2))\varphi +\int_{\Omega}\varphi (d\mu_1-d\mu_2).$$
		We see that $(u_1-u_2)$ is a very weak solution to the problem
		\begin{eqnarray}
			-\Delta(u_1-u_2) &=& f(h(u_1)-h(u_2)) + (\mu_1-\mu_2)	
			~\text{in}~\Omega,\nonumber\\
			u_1-u_2 &=& 0 ~\text{on}~ \partial\Omega.\nonumber	
		\end{eqnarray}
	Using the Proposition 1.5.4 (Kato's inequality) of  $\cite{Marcus}$, we can observe that
	\begin{equation}\label{l}
	-\int_{\Omega}(u_1-u_2)^+\Delta\varphi \leq \int_{\Omega}f(h(u_1)-h(u_2)) (sign_+ (u_1-u_2))\varphi +\int_{\Omega}\varphi (d\mu_1-d\mu_2),
	\end{equation}
		where, $sign_+ (u_1-u_2)=\chi_{\{x\in\Omega:u_1(x)\geq u_2(x)\}}$.
		Let us consider a $\varphi_0$ such that $-\Delta\varphi_0=1$ in $\Omega$ and $\varphi_0=0$ on $\partial\Omega$. Now the inequality in $\eqref{l}$ becomes
		\begin{align}\label{kato1}
			& \int_{\Omega}(u_1-u_2)^+  \leq \int_{\Omega}f(h(u_1)-h(u_2))(sign_+(u_1-u_2))\varphi_0 +\int_{\Omega}\varphi_0 d(\mu_1-\mu_2).
		\end{align}
		Similarly, it is easy to obtain 
		\begin{align}\label{kato2}
			& \int_{\Omega}(u_2-u_1)^+  \leq \int_{\Omega}f(h(u_2)-h(u_1))(sign_+(u_2-u_1))\varphi_0 +\int_{\Omega}\varphi_0 d(\mu_2-\mu_1).
		\end{align}
		Equations $\eqref{kato1}$ and $\eqref{kato2}$ are our required Kato type inequalities.\\
		We will now prove that if $h$ is strictly decreasing then $\eqref{eqq0}$ has a unique very weak solution. For if $u_1,u_2$ are two very weak solutions corresponding to the same measure data then  from $\eqref{kato1}$ considered over $A=\{x\in\Omega :u_1(x)\geq u_2(x)\}$
		we have 
		\begin{align}\label{appkato1}
			& \int_{A}(u_1-u_2)^+ +\int_{A}f(h(u_2)-h(u_1))\varphi_0 \leq  0.
		\end{align}
		Similarly from $\eqref{kato2}$ considered over $B=\{x\in\Omega :u_2(x)\geq u_1(x)\}$ we have 
		\begin{align}\label{appkato2}
			& \int_{B}(u_2-u_1)^+ +\int_{B}f(h(u_1)-h(u_2))\varphi_0 \leq  0.
		\end{align}
		Since the first term in $\eqref{appkato1}$, $\eqref{appkato2}$ are nonnegative, we get $0\leq \int_{A}f(h(u_2)-h(u_1))\varphi_0\leq 0$ and $0\leq \int_{B}f(h(u_1)-h(u_2))\varphi_0 \leq 0$. This implies that $h(u_1)=h(u_2)$ a.e. in $\Omega$. Due to the strictly decreasing nature of $h$, we conclude that $u_1=u_2$ a.e. in $\Omega$.
		%Now as $f>0$ and $\varphi_0\geq0$ we get $$\int_{\Omega}(u_1-u_2)^+ + \int_{\Omega}f(h(u_2)-h(u_1))\varphi_0  \leq\int_{\Omega}\varphi_0 d(\mu_1-\mu_2).  $$
		%So if $\mu_1$ becomes equal to $\mu_2$, then
		%\begin{align}
		%	\int_{\Omega}f(h(u_2)-h(u_1))\varphi_0 \leq 0.\nonumber
		%\end{align}
	%since $h(u_2)-h(u_1)\geq0$ for $u_1
	%	\geq u_2$. So we reach to the conclusion that $h(u_1)=h(u_2)$.
		\section*{Acknowledgement}
		Two of the authors, A. Panda and S. Ghosh, thanks for the financial assistantship received to carry out this research work from the Ministry of Human Resource Development(M.H.R.D.), Govt. of India and the Council of Scientific and Industrial Research(C.S.I.R.), Govt. of India respectively. This is also to declare that there are no financial conflict of interest whatsoever. Finally the	 authors thank the anonymous referee for the constructive comments and suggestions.

	\end{document}